\documentclass[a4paper,11pt]{article}
\usepackage[latin1]{inputenc}
\usepackage[english]{babel}
\usepackage{amsmath}
\usepackage{amsfonts}
\usepackage{amssymb}
\usepackage{epsfig}
\usepackage{amsopn}
\usepackage{amsthm}
\usepackage{color}
\usepackage{graphicx}
\usepackage{subfigure}
\usepackage{enumerate}
\usepackage{tikz-cd}
\newtheorem{theorem}{Theorem}[section]

\newtheorem{lemma}[theorem]{Lemma}

\newtheorem{definition}[theorem]{Definition}

\newtheorem*{theorem*}{Theorem}
\newtheorem*{lemma*}{Lemma}
\newtheorem*{remark*}{Remark}
\newtheorem*{definition*}{Definition}
\newtheorem*{proposition*}{Proposition}
\newtheorem*{corollary*}{Corollary}
\numberwithin{equation}{section}
\newcommand{\vertiii}[1]{{\left\vert\kern-0.25ex\left\vert\kern-0.25ex\left\vert #1
    \right\vert\kern-0.25ex\right\vert\kern-0.25ex\right\vert}}

\newcommand{\real}{\mathbb{R}}

\def\qed{\,\unskip\kern 6pt \penalty 500
\raise -2pt\hbox{\vrule \vbox to8pt{\hrule width 6pt
\vfill\hrule}\vrule}\par}
\definecolor{darkblue}{rgb}{0.05, .05, .65}
\definecolor{darkgreen}{rgb}{0.1, .65, .1}
\definecolor{darkred}{rgb}{0.8,0,0}
\newcommand{\beqn}{\begin{equation}}
\newcommand{\eeqn}{\end{equation}}
\newcommand{\bear}{\begin{eqnarray}}
\newcommand{\eear}{\end{eqnarray}}
\newcommand{\bean}{\begin{eqnarray*}}
\newcommand{\eean}{\end{eqnarray*}}
\begin{document}


\title{\huge \bf Optimal existence, uniqueness and blow-up for a quasilinear diffusion equation with spatially inhomogeneous reaction}

\author{
\Large Razvan Gabriel Iagar\,\footnote{Departamento de Matem\'{a}tica
Aplicada, Ciencia e Ingenieria de Materiales y Tecnologia
Electr\'onica, Universidad Rey Juan Carlos, M\'{o}stoles,
28933, Madrid, Spain, \textit{e-mail:} razvan.iagar@urjc.es},\\
[4pt] \Large Marta Latorre\,\footnote{Departamento de Matem\'{a}tica
Aplicada, Ciencia e Ingenieria de Materiales y Tecnologia
Electr\'onica, Universidad Rey Juan Carlos, M\'{o}stoles,
28933, Madrid, Spain, \textit{e-mail:} marta.latorre@urjc.es},
\\[4pt] \Large Ariel S\'{a}nchez,\footnote{Departamento de Matem\'{a}tica
Aplicada, Ciencia e Ingenieria de Materiales y Tecnologia
Electr\'onica, Universidad Rey Juan Carlos, M\'{o}stoles,
28933, Madrid, Spain, \textit{e-mail:} ariel.sanchez@urjc.es}\\
[4pt] }
\date{}
\maketitle

\begin{abstract}
Well-posedness and a number of qualitative properties for solutions to the Cauchy problem for the following nonlinear diffusion equation with a spatially inhomogeneous source
$$
\partial_tu=\Delta u^m+|x|^{\sigma}u^p,
$$
posed for $(x,t)\in\real^N\times(0,T)$, with exponents $1<p<m$ and $\sigma>0$, are established. More precisely, we identify the \emph{optimal class of initial conditions} $u_0$ for which (local in time) existence is ensured and prove \emph{non-existence of solutions} for the complementary set of data. We establish then (local in time) \emph{uniqueness and a comparison principle} for this class of data. We furthermore prove that any non-trivial solution to the Cauchy problem \emph{blows up in a finite time} $T\in(0,\infty)$ and \emph{finite speed of propagation} holds true for $t\in(0,T)$: if $u_0\in L^{\infty}(\real^N)$ is an initial condition with compact support and blow-up time $T>0$, then $u(t)$ is compactly supported for $t\in(0,T)$. We also establish in this work the \emph{absence of localization at the blow-up time} $T$ for solutions stemming from compactly supported data.
\end{abstract}

\


\smallskip


\section{Introduction}

The present work is aimed at establishing several qualitative properties of solutions to the Cauchy problem associated to a quasilinear reaction-diffusion equation
\begin{subequations}
\begin{equation}\label{eq1}
\partial_tu=\Delta u^m+|x|^{\sigma}u^p, \qquad (x,t)\in\real^N\times(0,T),
\end{equation}
\begin{equation}\label{ic}
u(x,0)=u_0(x), \qquad x\in\real^N,
\end{equation}
\end{subequations}
posed in dimension $N\geq1$ and with exponents $1<p<m$, $\sigma>0$, although some results will also work for the limiting exponent $p=1$. The more general class of initial conditions we work with is
\begin{equation}\label{icond}
u_0\in C(\real^N)\cap L^{\infty}(\real^N), \qquad u_0(x)\geq 0 \ {\rm for \ any} \ x\in\real^N, \qquad u_0\not\equiv0,
\end{equation}
where $C(\real^N)$ means the class of continuous functions in $\real^N$. Eq. \eqref{eq1} features an interesting competition between the diffusion term and the inhomogeneous source term. The former has a conservative effect on the total mass of the solution, spreading it throughout the space $\real^N$ during the evolution, while the latter leads to an increase of the $L^1$ norm of non-negative solutions. Together with this, let us notice that the source term weighs more over sets lying at positive distance from the origin, where $|x|^{\sigma}$ is large, and could be formally seen as negligible at the origin. All these imbalances enhance the mathematical interest (and also difficulty) of the study of the properties of the solutions to Eq. \eqref{eq1}.

Equations such as Eq. \eqref{eq1} with $\sigma=0$ have been a permanent object of study in the last half century, and by now there are two monographs describing well most of the features of their solutions, namely \cite{QS} for the semilinear case $m=1$ and \cite{S4} for the quasilinear case $m>1$. In particular, one of the most important properties of the solutions to reaction-diffusion equations is their possibility to blow up in finite time. By finite time blow-up we understand that, given an initial condition $u_0\in L^{\infty}(\real^N)$, there exists a time $T\in(0,\infty)$, called blow-up time, such that $u(t)\in L^{\infty}(\real^N)$ for any $t\in(0,T)$, but $u(T)\notin L^{\infty}(\real^N)$. In connection with this mathematical phenomenon, a number of questions have been raised (and at least partially solved in the homogeneous case $\sigma=0$):
\begin{itemize}
  \item for which initial conditions does the blow-up occur, stemming from the classical work by Fujita \cite{Fu66}, and later answered by \cite{Pi97, Pi98, Qi98, Su02} for weighted source terms.
  \item assuming that finite time blow-up occurs, at what points does this happen (the blow-up set) and with which time scale (the blow-up rate), see for example \cite{QS, S4, CdPE98, CdPE02, AT05, FdPV06, FdP22}.
  \item large time behavior near the blow-up time, which means, establishing the profiles (usually in self-similar form) to which general classes of solutions converge as $t\to T$, see for example \cite{QS, S4, GV97, Su03, FdPV06}.
  \item continuation of solutions after the blow-up time, see for example \cite{GV97}.
  \item establishing blow-up solutions with unusual or unexpected rates, a phenomenon known nowadays as blow-up of type II, see for example \cite{HV94, QS, GV97, MM09, MS21}.
\end{itemize}
All these questions show that the problems related to the mathematical phenomenon of finite time blow-up are very rich and interesting to study, but also, by inspecting the proofs, one can see that they are in many cases rather difficult.

Before stating our main results, we give below a short presentation of previous results for equations such as Eq. \eqref{eq1}, including a weighted source term. The question of local (in time) well-posedness of solutions is dealt with, up to our knowledge, for the first time by Baras and Kersner \cite{BK87}, where optimal classes of data for existence and non-existence of solutions are established for $m=1$ and a general weight $a(x)$ instead of $|x|^{\sigma}$. Shortly after, Andreucci and DiBenedetto \cite{AdB91} performed a very deep analysis of the properties related to well-posedness, regularity for short times and initial trace for solutions to the Cauchy problem for an equation similar to Eq. \eqref{eq1} but with a weight $(1+|x|)^{\sigma}$ and for any $\sigma\in\real$. Going back to the semilinear case $m=1$, the Fujita exponent $p_F(1,\sigma)=1+(\sigma+2)/N$ limiting between the range $1<p\leq p_F(1,\sigma)$ in which all non-trivial solutions blow up in finite time, and the range $p>p_F(1,\sigma)$ when global solutions also exist, has been obtained in \cite{Pi97, Pi98}, even for more general weights than the pure powers $|x|^{\sigma}$. Later on, a sequence of papers \cite{GLS, GS11, GLS13, GS18} answered one of the most natural questions related to the blow-up sets, that is, whether the origin can be a blow-up point or not for solutions to Eq. \eqref{eq1} with $m=1$. A classification of self-similar solutions to Eq. \eqref{eq1} for $m=1$ and depending on the magnitude of $p$ has been given in \cite{FT00}, while Mukai and Seki \cite{MS21} proved that for large values of $p>1$, blow-up of type II, with an infinite sequence of possible different blow-up rates, occurs, by constructing explicit examples.

In the quasilinear case $m>1$, the Fujita exponent $p_F(m,\sigma)=m+(\sigma+2)/N$ has been obtained by Qi \cite{Qi98} and extended by Suzuki \cite{Su02} to more general weights $K(x)$ instead of $|x|^{\sigma}$. Let us stress here that both of these papers \emph{only deal with the range $p>m$}, thus proving that any non-trivial solution blows up in finite time provided $m<p\leq p_F(m,\sigma)$ and that there exist global solutions for $p>p_F(m,\sigma)$, but without considering exponents $p\in(1,m)$ or $p=m$. Moreover, in \cite{Su02}, a second critical exponent, related to the borderline between blow-up solutions and global solutions based on the decay of the initial condition $u_0(x)$ as $|x|\to\infty$, is also identified. Shortly after, Andreucci and Tedeev established in \cite[Theorem 1.3]{AT05} blow-up rates for solutions to Eq. \eqref{eq1} also for $p>m$, limited (in our opinion, by technical reasons) to some finite interval for $\sigma>0$. In a different direction, a number of works \cite{FdPV06, KWZ11, Liang12, FdP22} (see also references therein) addressed the previous questions to porous medium equations with a spatially inhomogeneous source, but with compactly supported weights $a(x)$ instead of $|x|^{\sigma}$. Many problems related to the blow-up set, rate and asymptotic behavior are solved in this case, but the techniques make strong use of the fact that the weight is bounded, which allows to employ strong regularity results in H\"{o}lder spaces and obtain useful estimates for the solutions. Such techniques, at least at a first sight, cannot be easily extended to unbounded weights such as $|x|^{\sigma}$. We stress here that some of the quoted works (and many others in the semilinear case, stemming from the Hardy inequality) also dealt with values $\sigma<0$, but we do not extend the discussion in this direction here.

We thus notice that in the range of exponents $1<p<m$, even the basic results related to finite time blow-up are missing from literature. Being aware of the importance of the self-similar solutions for the study of nonlinear diffusion equations and in particular for Eq. \eqref{eq1}, the authors performed a classification of them in the range $1\leq p<m$ in \cite{IS19, IS21, ILS23}, the former works dealing with dimension $N=1$ and the latter one with any dimension $N\geq2$. It has been noticed that, if $1<p<m$ and $\sigma>0$, there exists at least one compactly supported self-similar solution, and its blow-up set depends on the magnitude of $\sigma$: for $\sigma>0$ sufficiently small the blow-up set is $\real^N$, while for $\sigma$ large a less common phenomenon known as blow-up at infinity (in the sense of \cite{La84, GU06}) holds true. This means that, for any $x\in\real^N$, $u(t,x)<\infty$ even at $t=T$, while
$$
M(t):=\|u(t)\|_{\infty}=u(x(t),t)\to\infty, \qquad |x(t)|\to\infty, \qquad {\rm as} \ t\to T.
$$
Together with these self-similar solutions, we have classified all the possible behaviors of self-similar subsolutions and supersolutions to Eq. \eqref{eq1}, something that will be very useful in the sequel. Let us also stress here that self-similar solutions to Eq. \eqref{eq1} with $p\in(1,m)$ have been classified also for negative values of $\sigma\in[-2,0)$ in \cite{ILS23, IS23, IMS23}, but the outcome is very different and the same questions for singular weights will not be addressed in the present paper.


\medskip

\noindent \textbf{Main results.} Assume from now on that $1<p<m$ and $\sigma>0$. In order to state our main results, we have to make first clear the notion of solution employed throughout the paper, which is the standard weak solution (see for example \cite{Su02}). Let us introduce first the following notation that will be used throughout the paper:
$$
Q_T:=\real^N\times(0,T).
$$
Moreover, for a fixed $t>0$ we will write $u(t)$ for the mapping $x\mapsto u(x,t)$. For the sake of completeness, we give the definition of a weak solution below.
\begin{definition}[Weak solution]\label{def.weak}
Given $u_0$ as in \eqref{icond}, we say that a function $u$ is a \emph{weak solution} to the Cauchy problem \eqref{eq1}-\eqref{ic} in $Q_T$ if $u\in C(\real^N\times(0,t))$ for any $t\in(0,T)$ and if, for any bounded domain $\Omega\subset\real^N$ and for any test function $\varphi\in C^{2,1}(\overline{\Omega}\times[0,T))$ such that $\varphi=0$ on $\partial\Omega\times(0,T)$, we have the following equality
\begin{equation}\label{weaksol}
\begin{split}
\int_{\Omega}u(x,t)\varphi(x,t)\,dx&-\int_{\Omega}u_0(x)\varphi(x,0)\,dx\\&=\int_0^t\int_{\Omega}\left(u(x,s)\partial_s\varphi(x,s)+u^m(x,s)\Delta\varphi(x,s)+|x|^{\sigma}u^p(x,s)\right)\,dx\,ds\\
&-\int_0^t\int_{\partial\Omega}u^m(x,s)\partial_{n}\varphi(x,s)\,dS\,ds,
\end{split}
\end{equation}
for any $t\in(0,T)$. We say that $u$ is a weak supersolution to the Cauchy problem \eqref{eq1}-\eqref{ic} if we replace equality by the sign $\geq$ in \eqref{weaksol} and we say that $u$ is a weak subsolution if we replace equality by the sign $\leq$ in \eqref{weaksol}.
\end{definition}
We enumerate below our main results.

\medskip

Our first theorem establishes sharp classes of initial conditions $u_0(x)$ such that the Cauchy problem \eqref{eq1}-\eqref{ic} admits or does not admit a weak solution in $Q_T$ for some $T>0$. Such a solution will be called \emph{local in time}, since it is a priori defined on some finite time interval $(0,T)$.
\begin{theorem}[Sharp existence and non-existence]\label{th.exist}
Let $u_0$ be as in \eqref{icond}. Then:

(a) If there exist some $K>0$ and some $R>0$ such that
\begin{equation}\label{decay}
u_0(x)\leq K|x|^{-\sigma/(p-1)}, \qquad {\rm for \ any} \ x\in\real^N, \ |x|\geq R,
\end{equation}
then there exist $T>0$ and a solution $u$ to the Cauchy problem \eqref{eq1}-\eqref{ic} defined and bounded in $Q_T$.

(b) On the contrary, if
\begin{equation}\label{not.decay}
\liminf\limits_{|x|\to\infty}|x|^{\sigma/(p-1)}u_0(x)=+\infty,
\end{equation}
then the Cauchy problem \eqref{eq1}-\eqref{icond} does not admit any (local in time) solution.
\end{theorem}
This theorem is practically optimal, as it gives the threshold (in terms of decay as $|x|\to\infty$) of the initial condition in order to ensure (at least for a short time) the existence of a solution to the Cauchy problem. Let us remark at this point that the decay \eqref{decay} as $|x|\to\infty$ has appeared in our previous papers devoted to self-similar solutions as a limiting behavior and as the unique possible tail behavior for self-similar profiles, suggesting heuristically that there might be no possible solution with a slower decay rate. The analysis we perform in this work confirms this limiting property of it.

The next result addresses the other very basic question that one can formulate with respect to the Cauchy problem, that is, the uniqueness of a solution (once we know when it exists). This does not follows directly from standard theory of parabolic equations, since the coefficient $|x|^{\sigma}$ is not bounded as $|x|\to\infty$ and thus it requires an ad-hoc proof.
\begin{theorem}[Uniqueness and comparison principle]\label{th.uniq}
Let $T>0$ and let $u$, $v$ to be respectively a subsolution and a supersolution to Eq. \eqref{eq1} in $Q_T$ such that $u(x,0)\leq v(x,0)$ for any $x\in\real^N$. Assume that there exist a continuous function $M:(0,T)\mapsto(0,\infty)$ and $R>0$ sufficiently large such that
\begin{equation}\label{bound}
\sup\limits_{s\in(0,t)}\{u(x,s),v(x,s)\}\leq M(t)|x|^{-\sigma/(p-1)},
\end{equation}
for any $t\in(0,T)$ and $x\in\real^N$ with $|x|\geq R$. Then $u\leq v$ in $Q_T$. In particular, if $u_0$ is an initial conditions satisfying \eqref{icond} and \eqref{decay}, then there exists a unique solution to the Cauchy problem \eqref{eq1}-\eqref{ic}.
\end{theorem}
A comparison principle becomes a powerful tool in conjunction with the classification of subsolutions and supersolutions in self-similar form, allowing us to obtain some new properties of solutions. One of them is the fact that any non-trivial solution to Eq. \eqref{eq1} blows up in finite time.
\begin{theorem}[Finite time blow-up]\label{th.BU}
Let $u_0$ be an initial condition as in \eqref{icond} and satisfying \eqref{decay}. Then the unique solution to the Cauchy problem \eqref{eq1}-\eqref{ic} blows up in a finite time $T\in(0,\infty)$ in the sense that $u(t)\in L^{\infty}(\real^N)$ for any $t\in(0,T)$, but $u(T)$ becomes unbounded. The same holds true for $p=1$ and $\sigma>0$.
\end{theorem}
At a first sight, this result might look as an expected one: indeed, $p\in[1,m)$ is in any case smaller than the Fujita exponent $p_{F}(m,\sigma)=m+(\sigma+2)/N$. But, as it was shown in some previous works, things are not always so easy: actually, in some neighboring cases, this is no longer true. If we let $\sigma\in(-2,0)$, it was recently established in \cite{IMS23} that for $p\in(1,p_*)$, with
$$
p_*=1-\frac{\sigma(m-1)}{2}\in(1,m),
$$
there are always global solutions. A similar fact holds true if the pure power weight $|x|^{\sigma}$ is replaced by a localized one: in dimension $N=1$, it has been established in \cite{FdPV06} that there exists a critical exponent $p_0=(m+1)/2\in(1,m)$ such that for any $p\in(1,p_0)$ every solution is global in time. Later, it was proved in \cite{Liang12} that in higher dimensions $N\geq2$, and with localized weighted source, for any $p\in(1,m)$, every solution is global in time. Moreover, notice that finite time blow-up of any non-trivial solution also holds true for $p=1$ and $\sigma>0$, in strong contrast with the homogeneous case $p=1$ and $\sigma=0$ where solutions are global in time.

We can also employ the comparison principle and our knowledge of self-similar solutions stemming from \cite{IS21, ILS23} in order to obtain some properties related to the support of general solutions with compactly supported initial conditions.
\begin{theorem}[Finite speed of propagation]\label{th.fsp}
Let $u_0$ be as in \eqref{icond} and such that ${\rm supp}\,u_0\subset B(0,R)$ for some $R>0$. Let $u$ be the unique solution to the Cauchy problem \eqref{eq1}-\eqref{ic} and $T\in(0,\infty)$ its finite blow-up time. Then, for any $t\in(0,T)$, there exists $R(t)<\infty$ such that ${\rm supp}\,u(t)\subset B(0,R(t))$ for any $t\in(0,T)$.
\end{theorem}
This result opens up for another interesting question: is the support of at least some non-trivial solutions localized? That is, if $T>0$ is the blow-up time of a solution to the Cauchy problem \eqref{eq1}-\eqref{ic} (which exists according to Theorem \ref{th.BU}), is it true that there exists a sufficiently large $R>0$ such that ${\rm supp}\,u(t)\subset B(0,R)$ for any $t\in[0,T]$? An inspection of the self-similar solutions constructed in \cite{ILS23} suggests that this assertion should be false, and indeed, our last result states it in general.
\begin{theorem}[Absence of localization]\label{th.loc}
Let $u$ be a solution to the Cauchy problem \eqref{eq1}-\eqref{ic} with an initial condition $u_0$ satisfying \eqref{icond} and with compact support and having a blow-up time $T\in(0,\infty)$. Let
$$
\zeta(t):=\sup\{|x|: u(x,t)>0\}, \qquad t\in(0,T),
$$
be the edge of the support of $u(t)$. Then $\lim\limits_{t\to T}\zeta(t)=\infty$.
\end{theorem}

\medskip

The rest of the paper is devoted to the proofs of the previous theorems, in the order they were stated. Every theorem will be proved in a separate section, with the exception of Theorem \ref{th.exist}, whose proof will be split into two sections, since the techniques employed in the proofs of the existence and the non-existence results are totally different.

\section{Existence of solutions}\label{sec.exist}

This (rather short) section is devoted to the proof of Theorem \ref{th.exist}, part (a). To this end, we proceed by approximation, and this is performed in the following
\begin{lemma}\label{lem.approx}
Let $u_0$ be an initial condition as in \eqref{icond}. Assume that there exists a supersolution $v(x,t)$ to Eq. \eqref{eq1} in $Q_T$ such that $u_0(x)\leq v(x,0)$ for any $x\in\real^N$. Then, there exists a weak solution $u$ to the Cauchy problem \eqref{eq1}-\eqref{ic} in $Q_T$ such that $u(x,t)\leq v(x,t)$ for any $(x,t)\in Q_T$.
\end{lemma}
\begin{proof}[Sketch of the proof]
The proof is given in \cite[Proposition 2.8]{Su02}, where the restriction $p>m$ is assumed, but a simple inspection of the proof shows that the proof holds true for any $p$ and $m>1$. We give a short sketch, for the sake of completeness. The idea is to construct a solution by approximation with solutions to a sequence of Dirichlet problems in large balls. Let $u_{0,n}\in C_0^{\infty}(B(0,n))$ be such that $u_{0,n}\leq u_0$ in $B(0,n)=\{x\in\real^N: |x|<n\}$, for $n$ natural number, $u_{0,n}\to u_0$ uniformly on compact subsets of $\real^N$, and $u_{0,n}(x)\leq u_{0,n+1}(x)$ for any $x\in B(0,n)$. Let then $u_n$ be the unique solution to the homogeneous Dirichlet problem in $B(0,n)$ for Eq. \eqref{eq1} with initial condition $u_{0,n}$. Since this problem is a regular one, we infer by the standard comparison principle (see for example \cite[Proposition 2.2]{Su02}) that
$$
u_n(x,t)\leq u_{n+1}(x,t)\leq v(x,t), \qquad {\rm for \ any} \ (x,t)\in B(0,n)\times(0,T),
$$
and for any $n\in\mathbb{N}$, thus, there exists $u(x,t)=\lim\limits_{n\to\infty}u_n(x,t)\leq v(x,t)$. On the one hand, the continuity of the limit solution $u$ in $Q_T$ follows from the existence of an uniform modulus of continuity for the approximant family $u_n$ stemming from \cite[Proposition 1 and Theorem 1]{DiB83}. On the other hand, we readily get that the limit $u(x,t)$ satisfies the weak formulation \eqref{weaksol} for the Cauchy problem \eqref{eq1}-\eqref{ic} by Lebesgue's monotone convergence theorem. More details can be found in the proof of \cite[Proposition 2.8]{Su02}.
\end{proof}
We thus infer from Lemma \ref{lem.approx} that, in order to prove the existence of local in time solutions to the Cauchy problem \eqref{eq1}-\eqref{ic}, it is sufficient to construct a supersolution. This is done by using previous results related to a similar equation studied by Andreucci and DiBenedetto in \cite{AdB91}.
\begin{proof}[Proof of Theorem \ref{th.exist}, part (a)]
Let $u_0$ be an initial condition satisfying \eqref{icond} and \eqref{decay}. Following \cite{AdB91} and introducing first
$$
B_r(x):=\{y\in\real^N: |x-y|<(1+|x|)^r\}, \qquad r:=-\frac{\sigma(m-1)}{2(p-1)}<0,
$$
and then the norm
\begin{equation}\label{norm3}
\vertiii{u_0}:=\sup\limits_{x\in\real^N}(1+|x|)^{\sigma/(p-1)}\left[\frac{1}{|B_r(x)|}\int_{B_r(x)}u_0(y)\,dy\right],
\end{equation}
we readily observe that for very large $x$, for example $|x|\geq R+1$, with $R$ given in the statement of the condition \eqref{decay}, the right hand side of \eqref{norm3} is bounded due to the bound \eqref{decay} on $u_0$ and easy calculations. We conclude that $\vertiii{u_0}<\infty$ if the condition \eqref{decay} is fulfilled, and we are thus in the hypothesis of \cite[Theorem 3.1]{AdB91} with $q=1$ in the notation therein, which ensures (since $p<m<m+2/N$) the existence of a time $T>0$ and of a solution $\overline{u}$ in $Q_T$ to the Cauchy problem
\begin{equation}\label{eq2}
\left\{\begin{array}{ll}\partial_t\overline{u}=\Delta\overline{u}^m+(1+|x|)^{\sigma}\overline{u}^p, & (x,t)\in\real^N\times(0,T), \\ \overline{u}(x,0)=u_0(x), & x\in\real^N.\end{array}\right.
\end{equation}
Since $\sigma>0$, we also observe that $\overline{u}$ constructed above is a supersolution to the Cauchy problem \eqref{eq1}-\eqref{ic}, and Lemma \ref{lem.approx} then establishes the existence of at least a solution to the Cauchy problem \eqref{eq1}-\eqref{ic} in $Q_T$, as claimed.
\end{proof}

\section{Non-existence of solutions}\label{sec.nonexist}

This section is dedicated to the proof of the non-existence part in Theorem \ref{th.exist}. The argument adapts a classical one employed for the semilinear case $m=1$ by Baras and Kersner \cite{BK87}. For the convenience of the reader, we give below a detailed proof.
\begin{proof}[Proof of Theorem \ref{th.exist}, part (b)]
Let $u_0$ be as in \eqref{icond} and such that \eqref{not.decay} holds true. Assume for contradiction that there exists some $T>0$ and a solution $u$ to the Cauchy problem \eqref{eq1}-\eqref{ic} in $Q_T$. By eventually decreasing a bit $T$ (in order to avoid taking exactly the blow-up time, if any), we may assume without loss of generality that
$$
M(T):=\sup\limits_{t\in(0,T)}\|u(t)\|_{\infty}<\infty.
$$
Letting $\varphi\in C_0^{\infty}(Q_T)$ be a test function such that $\varphi\geq0$ in $Q_T$ and $\varphi(x,T)=0$ for any $x\in\real^N$, we infer from the weak formulation that
\begin{equation*}
\begin{split}
\int_{\real^N}u_0(x)\varphi(x,0)\,dx&=\int_0^T\int_{\real^N}\left(-u\varphi_t-u^m\Delta\varphi-|x|^{\sigma}u^p\varphi\right)\,dx\,dt\\
&=\int_0^T\int_{\real^N}\left(-u\frac{\varphi_t}{|x|^{\sigma}\varphi}-u^m\frac{\Delta\varphi}{|x|^{\sigma}\varphi}-u^p\right)|x|^{\sigma}\varphi\,dx\,dt\\
&\leq\int_0^T\int_{\real^N}\left(u\frac{(-\varphi_t)_+}{|x|^{\sigma}\varphi}+u^m\frac{(-\Delta\varphi)_+}{|x|^{\sigma}\varphi}-u^p\right)|x|^{\sigma}\varphi\,dx\,dt\\
&\leq\int_0^T\int_{\real^N}\left(u\frac{(-\varphi_t)_++M(T)^{m-1}(-\Delta\varphi)_+}{|x|^{\sigma}\varphi}-u^p\right)|x|^{\sigma}\varphi\,dx\,dt,
\end{split}
\end{equation*}
where $(\cdot)_+$ designs the positive part, and for the last inequality we have used the fact that $u^m\leq M(T)^{m-1}u$ in $Q_T$. By following \cite[Proposition 1 (i)]{BK87} and optimizing in $u$ in the right hand side of the above inequality, we find
\begin{equation}\label{interm1}
\begin{split}
\int_{\real^N}u_0(x)\varphi(x,0)\,dx&\leq\frac{p-1}{p^{p/(p-1)}}\int_0^T\int_{\real^N}\left[(-\varphi_t)_++M(T)^{m-1}(-\Delta\varphi)_+\right]^{p/(p-1)}\\
&\times|x|^{-\sigma/(p-1)}\varphi^{-1/(p-1)}\,dx\,dt.
\end{split}
\end{equation}
We next particularize $\varphi$ as a separate variable function, by setting
\begin{equation}\label{interm2}
\varphi(x,t)=\left[1-\frac{t}{T}\right]^{p/(p-1)}\psi(x),
\end{equation}
with $\psi\in C_0^{\infty}(\real^N)$, $\psi\geq0$. Since $\varphi(x,0)=\psi(x)$, a direct substitution of the ansatz \eqref{interm2} into the right hand side of \eqref{interm1} gives
\begin{equation*}
\begin{split}
\int_{\real^N}u_0(x)\psi(x)\,dx&\leq\frac{p-1}{p^{p/(p-1)}}\int_0^T\int_{\real^N}\left[\frac{p\psi(x)}{(p-1)T}+M(T)^{m-1}\left(1-\frac{t}{T}\right)(-\Delta\psi(x))_+\right]^{p/(p-1)}\\
&\times|x|^{-\sigma/(p-1)}\psi(x)^{-1/(p-1)}\,dx,dt\\
&\leq\frac{(p-1)T}{p^{p/(p-1)}}\int_{\real^N}\left[\frac{p\psi(x)}{(p-1)T}+M(T)^{m-1}\left(1-\frac{t}{T}\right)(-\Delta\psi(x))_+\right]^{p/(p-1)}\\
&\times|x|^{-\sigma/(p-1)}\psi(x)^{-1/(p-1)}\,dx,dt,
\end{split}
\end{equation*}
which furthermore gives, by raising both sides to the power $(p-1)/p$ and applying Minkowski's inequality with power $p/(p-1)>1$ in the right hand side:
\begin{equation*}
\begin{split}
\left(\int_{\real^N}u_0(x)\psi(x)\,dx\right)&^{(p-1)/p}\leq\frac{[(p-1)T]^{(p-1)/p}}{p}\\
&\times\left[\left(\int_{\real^N}\left(\frac{p}{(p-1)T}\right)^{p/(p-1)}\psi(x)|x|^{-\sigma/(p-1)}\,dx\right)^{(p-1)/p}\right.\\
&\left.+\left(\int_{\real^N}\frac{M(T)^{(m-1)p/(p-1)}(-\Delta\psi(x))^{p/(p-1)}}{|x|^{\sigma/(p-1)}\psi(x)^{1/(p-1)}}\,dx\right)^{(p-1)/p}\right]\\
&=[(p-1)T]^{-1/p}\left(\int_{\real^N}\frac{\psi(x)}{|x|^{\sigma/(p-1)}}\right)^{(p-1)/p}+\frac{[(p-1)T]^{(p-1)/p}M(T)^{m-1}}{p}\\
&\times\left[\int_{\real^N}|x|^{-\sigma/(p-1)}(-\Delta\psi(x))^{p/(p-1)}\psi(x)^{-1/(p-1)}\,dx\right]^{(p-1)/p}.
\end{split}
\end{equation*}
Fix now a radially symmetric function $\eta\in C_0^{\infty}(\real^N)$ such that $\eta(r)\geq0$ for any $r=|x|\geq0$, and ${\rm supp}\,\eta\subseteq\overline{B(0,2)}\setminus B(0,1)$. Since $\eta$ is a convex function in a neighborhood of its edges of the support, it follows readily that there exists a sufficiently large $k>0$ such that
\begin{equation}\label{interm0}
-\Delta\eta\leq k\eta.
\end{equation}
Set, for $\lambda>0$ fixed, $\psi(x):=\eta(x/\lambda)$ as test function in the above estimates. On the one hand, we obtain by direct substitution and also employing \eqref{interm0} in the second term of the right hand side that
\begin{equation}\label{interm3}
\begin{split}
\left(\int_{\real^N}u_0(x)\psi(x)\,dx\right)&^{(p-1)/p}\leq[(p-1)T]^{-1/p}\left(\int_{\real^N}\frac{\psi(x)}{|x|^{\sigma/(p-1)}}\right)^{(p-1)/p}\\
&+\frac{[(p-1)T]^{(p-1)/p}M(T)^{m-1}k}{p\lambda^2}\left(\int_{\real^N}\frac{\psi(x)}{|x|^{\sigma/(p-1)}}\right)^{(p-1)/p}.
\end{split}
\end{equation}
On the other hand, noticing that $\psi(x)$ is supported for $|x|\geq\lambda$, we can write
\begin{equation}\label{interm4}
\begin{split}
\int_{\real^N}u_0(x)\psi(x)\,dx&=\int_{\real^N}u_0(x)|x|^{\sigma/(p-1)}\frac{\psi(x)}{|x|^{\sigma/(p-1)}}\,dx\\
&\geq\inf\limits_{|x|\geq\lambda}\left(u_0(x)|x|^{\sigma/(p-1)}\right)\int_{\real^N}\frac{\psi(x)}{|x|^{\sigma/(p-1)}}\,dx.
\end{split}
\end{equation}
We then infer from gathering \eqref{interm3} and \eqref{interm4} and simplifying from both sides the integral term, that
\begin{equation}\label{interm5}
\inf\limits_{|x|\geq\lambda}\left(u_0(x)|x|^{\sigma/(p-1)}\right)\leq[(p-1)T]^{-1/p}+\frac{[(p-1)T]^{(p-1)/p}M(T)^{m-1}k}{p\lambda^2}.
\end{equation}
Since $\lambda>0$ has been arbitrarily chosen, we can pass to the limit as $\lambda\to\infty$ in \eqref{interm5} and finally deduce that
\begin{equation}\label{interm6}
\liminf\limits_{|x|\to\infty}|x|^{\sigma/(p-1)}u_0(x)\leq[(p-1)T]^{-1/p}.
\end{equation}
But we observe that \eqref{interm6} is in contradiction with \eqref{not.decay}. This contradiction implies the non-existence of solutions (even for a very short time) when \eqref{not.decay} is in force, as claimed.
\end{proof}

\section{Uniqueness and comparison}\label{sec.uniq}

This short section establishes the comparison principle stated in Theorem \ref{th.uniq} and is a very useful result in the sequel. More precisely, it will allow us to obtain finer properties of the solutions to the Cauchy problem \eqref{eq1}-\eqref{ic} by comparing general solutions of it to specific subsolutions and supersolutions.
\begin{proof}[Proof of Theorem \ref{th.uniq}]
Let $T>0$, a subsolution $u$ and a supersolution $v$ to Eq. \eqref{eq1} as in the statement of Theorem \ref{th.uniq}. By setting $g:=u-v$, $h:=u^p-v^p$, we can subtract the inequalities satisfied by $u$ and $v$ to find
\begin{equation}\label{interm7}
\partial_tg-\Delta(u^m-v^m)\leq|x|^{\sigma}h, \qquad (x,t)\in Q_T.
\end{equation}
We illustrate below a formal argument containing the ``core" of the proof. By multiplying \eqref{interm7} by the function ${\rm sign}_+(u-v)$, where
$$
{\rm sign}_+(x)=\left\{\begin{array}{ll}1, & {\rm if} \ x>0,\\ 0, & {\rm if} \ x\leq0,\end{array}\right.
$$
and then integrating over $\real^N$ and employing Kato's inequality (see for example \cite{BP08, Kato}), we have
\begin{equation*}
\begin{split}
\partial_t g\,{\rm sign}_+(u-v)&\leq\Delta(u^m-v^m)\,{\rm sign}_+(u-v)+|x|^{\sigma}(u^p-v^p)\,{\rm sign_+}(u-v)\\
&\leq\Delta(u^m-v^m)_{+}+|x|^{\sigma}h_+,
\end{split}
\end{equation*}
and we further infer by integration that
\begin{equation}\label{interm8}
\frac{d}{dt}\|g_+(t)\|_{1}\leq\int_{\real^N}|x|^{\sigma}h(x,t)_{+}\,dx.
\end{equation}
Noticing that \eqref{interm8} acts only in the region where $u>v$, as otherwise both positive parts are zero, we next employ the standard numerical inequality
$$
\frac{u^p-v^p}{u-v}\leq p(u^{p-1}+v^{p-1}), \qquad {\rm if} \ u>v \ {\rm and} \ p>1,
$$
together with \eqref{bound} to obtain
\begin{equation}\label{interm9}
\frac{d}{dt}\|g_+(t)\|_1\leq p\int_{\real^N}|x|^{\sigma}(u^{p-1}(x,t)+v^{p-1}(x,t))g_{+}(x,t)\,dx\leq 2pM(t)^{p-1}\|g_+(t)\|_{1},
\end{equation}
for any $t\in(0,T)$. An application of Gronwall's Lemma then entails
$$
\|g_+(t)\|_1\leq C(t)\|g_+(0)\|_1=0,
$$
since $u_0\leq v_0$ by the statement of Theorem \ref{th.uniq}. We infer that $g_+(t)=0$ a.\,e. and for any $t\in(0,T)$, which, together with the continuity of $u$ and $v$, gives that $u(x,t)\leq v(x,t)$ for any $(x,t)\in Q_T$. A rigorous proof, performed by employing a monotone increasing smooth approximation of the sign function and a cut-off function, follows practically verbatim the proof of the uniqueness result for the similar equation with the weight $(1+|x|)^{\sigma}$ (instead of $|x|^{\sigma}$) done in \cite[Section 15]{AdB91}.
\end{proof}

\section{Blow-up in finite time}\label{sec.BU}

This section is dedicated to the proof of Theorem \ref{th.BU}. The proof is based on the previously established comparison principle and the availability of suitable subsolutions in self-similar form.
\begin{proof}[Proof of Theorem \ref{th.BU}]
We look for subsolutions in backward self-similar form
\begin{equation}\label{SSS}
\underline{u}(x,t;T)=(T-t)^{-\alpha}f(\xi), \qquad \xi=|x|(T-t)^{\beta},
\end{equation}
with exponents $\alpha$ and $\beta$ given by
\begin{equation}\label{exp.SS}
\alpha=\frac{\sigma+2}{L}, \qquad \beta=\frac{m-p}{L}, \qquad L=\sigma(m-1)+2(p-1)>0,
\end{equation}
and with profiles $f$ solving the following differential equation:
\begin{equation}\label{SSODE}
(f^m)''(\xi)+\frac{N-1}{\xi}(f^m)'(\xi)-\alpha f(\xi)+\beta\xi f'(\xi)+\xi^{\sigma}f^p(\xi)=0.
\end{equation}
To be more specific, previous works by the authors established the existence of profiles $f(\xi)$ solutions to \eqref{SSODE} supported on a compact interval $[\xi_1,\xi_2]\subset(0,\infty)$ and having the following local behavior:

(a) $f(\xi_1)=0$, $(f^m)'(\xi_1)>0$,

(b) $f(\xi_2)=0$, $(f^m)'(\xi_2)=0$,

(c) $f(\xi)>0$ for any $\xi\in(\xi_1,\xi_2)$.

Indeed, the existence of such profiles is granted by \cite[Proposition 3.4]{IS21} in dimension $N=1$ and for $1<p<m$, \cite[Proposition 3.2]{IS19} in dimension $N=1$ and for $p=1$, and Steps 3 and 4 in \cite[Section 5]{ILS23} together with \cite[Lemma 3.4]{ILS23} in dimension $N\geq2$ and for $1\leq p<m$. Notice that the self-similar functions defined by \eqref{SSS} with arbitrary blow-up time $T>0$ and profile $f$ satisfying the conditions (a), (b) and (c) above (and extended by zero outside the interval $[\xi_1,\xi_2]$) are subsolutions to Eq. \eqref{eq1} (and in fact they would be true solutions, except for the contact point between $f(\xi)$ and zero at $\xi=\xi_1$). Moreover, since $T>0$ is a free parameter, the subsolution $\underline{u}(x,t;T)$ has as initial condition
\begin{equation}\label{init.SSS}
\begin{split}
&\underline{u}(x,0;T)=T^{-\alpha}f(|x|T^{\beta}), \qquad \|\underline{u}(\cdot,0;T)\|_{\infty}=T^{-\alpha}\max\{f(\xi):\xi\in[\xi_1,\xi_2]\},\\
&{\rm supp}\,\underline{u}(\cdot,0;T)=[\xi_1T^{-\beta},\xi_2T^{-\beta}].
\end{split}
\end{equation}

Let now $u_0$ be an initial condition as in \eqref{icond} and \eqref{decay} (which in particular includes the compactly supported case) and let $u$ be the unique solution to the Cauchy problem \eqref{eq1}-\eqref{ic} for $1<p<m$ following from Theorems \ref{th.exist} and \ref{th.uniq}. Assume for contradiction that $u$ is a global solution, that is, $u(t)\in L^{\infty}(\real^N)$ for any $t\in(0,\infty)$. It is then obvious that $u$ is a supersolution to the Cauchy problem for the porous medium equation
\begin{equation}\label{PME}
v_t=\Delta v^m, \qquad (x,t)\in\real\times(0,\infty), \qquad v(x,0)=u_0(x), \qquad x\in\real^N,
\end{equation}
and the comparison principle applied to the Cauchy problem \eqref{PME} entails that
$$
u(x,t)\geq v(x,t), \qquad (x,t)\in\real^N\times(0,\infty),
$$
where $v$ is the unique solution to \eqref{PME}. According to well-known properties of the porous medium equation (see for example \cite[Proposition 9.19]{VazquezPME}), the expansion of the support of $v(t)$ covers the whole space as $t\to\infty$, hence there exists $\tau_0>0$ such that
$$
B(0,2\xi_2)\subset{\rm supp}\,v(\tau_0)\subseteq{\rm supp}\,u(\tau_0)
$$
and there exists some $\epsilon>0$ such that
\begin{equation}\label{interm14bis}
u(x,\tau_0)\geq v(x,\tau_0)\geq\epsilon, \qquad x\in B(0,\xi_2).
\end{equation}
We next observe that \eqref{init.SSS} implies that, for any $T>1$, we have
\begin{equation}\label{interm15}
{\rm supp}\,\underline{u}(\cdot,0;T)=[\xi_1T^{-\beta},\xi_2T^{-\beta}]\subset B(0,\xi_2)\subset{\rm supp}\,u(\tau_0).
\end{equation}
Pick now some $T_0>1$ sufficiently large such that
\begin{equation}\label{interm15bis}
\|\underline{u}(\cdot,0;T_0)\|_{\infty}=T_0^{-\alpha}\max\{f(\xi):\xi\in[\xi_1,\xi_2]\}<\epsilon,
\end{equation}
and notice that \eqref{interm14bis}, \eqref{interm15} and \eqref{interm15bis} lead to the inequality
$$
u(x,\tau_0)\geq\underline{u}(x,0;T_0).
$$
The comparison principle then entails that
$$
u(x,\tau_0+t)\geq\underline{u}(x,t;T_0), \qquad (x,t)\in\real^N\times(0,\infty),
$$
and we reach a contradiction with the assumption of $u$ being a global solution, since $\underline{u}(\cdot,\cdot;T_0)$ blows up at $t=T_0$. This contradiction completes the proof. Notice that this proof also works for $p=1$.
\end{proof}

\section{Finite speed of propagation}\label{sec.FSP}

The goal of this section is to prove Theorem \ref{th.fsp}. The idea of the proof is to compare from above a compactly supported initial condition (or solution at some time $t\in(0,T)$) with a family of supersolutions in self-similar form built on basis of the classification of self-similar solutions to Eq. \eqref{eq1} performed in \cite{IS21, ILS23}. In these two works (the former in dimension $N=1$ and the latter in dimension $N\geq2$) we investigated the self-similar solutions to Eq. \eqref{eq1} in the same range $1<p<m$, in the form \eqref{SSS}, with exponents $\alpha$ and $\beta$ given in \eqref{exp.SS} and with profiles $f$ solving the differential equation \eqref{SSODE}. We recall here, as preliminary facts, statements granting the existence of the self-similar profiles that will be used throughout this section.
\begin{lemma}\label{lem.interf}
Let $N\geq2$. Then there exists $\Xi\in(0,\infty)$ such that, for any $\xi_0\in(0,\Xi)$, there exists a unique solution $f$ to \eqref{SSODE} which is decreasing for $\xi\in(0,\xi_0)$, having an interface exactly at $\xi=\xi_0$, in the sense that
\begin{equation}\label{interf}
f(\xi_0)=0, \qquad f(\xi)>0 \ {\rm for} \ \xi\in(0,\xi_0), \qquad (f^m)'(\xi_0)=0
\end{equation}
and having a vertical asymptote as $\xi\to0$ with the local behavior
\begin{equation}\label{vert.asympt}
f(\xi)\sim D\xi^{-(N-2)/m}, \ {\rm if} \ N\geq3, \qquad f(\xi)\sim D(-\ln\,\xi)^{1/m}, \ {\rm if} \ N=2,
\end{equation}
where $D>0$ is a constant. If $N=1$, the previous statement remains true with the addition that the local behavior \eqref{vert.asympt} as $\xi\to0$ is replaced by the following
\begin{equation}\label{neg.slope}
f(0)=a>0, \qquad f'(0)<0.
\end{equation}
\end{lemma}
\begin{proof}
For $N\geq2$, the conclusion follows from the proof of Theorem 1.4 given in \cite[Section 4]{ILS23}. More precisely, the orbits connecting the critical points $Q_5$ and $P_1$ of the phase space in the notation therein represent the two local behaviors \eqref{vert.asympt} as $\xi\to0$ (see \cite[Lemmas 3.2 and 3.3]{ILS23}), respectively \eqref{interf} as $\xi\to\xi_0$ (see \cite[Lemma 2.2]{ILS23}). Thus, \cite[Step 4, Section 4]{ILS23} implies the existence of the desired orbits and then profiles. In dimension $N=1$, the conclusion follows directly from \cite[Proposition 3.5]{IS21} and its proof, which also establishes that, if we denote by $a(\xi_0)$ the value of $f(0)$ in \eqref{neg.slope} for the profile with interface at $\xi=\xi_0$ as in \eqref{interf}, then $a(\xi_0)\to\infty$ as $\xi_0\to0$.
\end{proof}
\begin{lemma}\label{lem.Q1}
For any $a\in(0,\infty)$, there exists a unique solution $f$ to \eqref{SSODE} such that $f(0)=a$, $f'(0)=0$. This solution is increasing in a maximal right neighborhood $(0,\xi_1(a))$, reaching a local maximum at $\xi=\xi_1(a)$.
\end{lemma}
\begin{proof}
The existence and uniqueness follow from standard results in the theory of ODEs applied to \eqref{SSODE}, but a different proof based on a phase space argument is given in \cite[Lemma 3.1]{ILS23} in dimension $N\geq2$, respectively \cite[Lemma 2.5]{IS21} and the remark at the end of its proof in dimension $N=1$. Let us now fix such a solution $f$ with $f(0)=a$, $f'(0)=0$. Noticing that
$$
\lim\limits_{\xi\to0}\frac{(N-1)}{\xi}(f^m)'(\xi)=(N-1)(f^m)''(0),
$$
we readily get by taking limits as $\xi\to0$ in \eqref{SSODE} that
$$
N(f^m)''(0)-\alpha a=0, \qquad {\rm that \ is}, \qquad (f^m)''(0)=\frac{\alpha a}{N}>0.
$$
Thus, $f$ starts increasingly in a right neighborhood of $\xi=0$. Then the classification of all possible behaviors of profiles to \eqref{SSODE} given in \cite{IS21, ILS23} shows that, with $\sigma>0$, any profile $f(\xi)$ must have a local behavior leading to either $f(\xi)\to0$ as $\xi\to\infty$ or $f(\xi_0)=0$ for some $\xi_0\in(0,\infty)$. This shows that $f(\xi)$ cannot be increasing on $(0,\infty)$, and it suffices to let $\xi_1(a)$ be the first maximum point of $f$ in order to reach the conclusion.
\end{proof}
With these preliminaries, we are in a position to complete the proof of Theorem \ref{th.fsp}.
\begin{proof}[Proof of Theorem \ref{th.fsp}]
Let us assume first that $N\geq2$. Also assume that, for some time $t_0\in[0,\infty)$, $u(t_0)\in L^{\infty}(\real^N)$ and is a compactly supported function in $\real^N$. Let
$$
\zeta(t_0):=\sup\{|x|: x\in\real^N, u(x,t_0)>0\}
$$
be the maximum amplitude of the positivity set of $u(t_0)$ and $a=\|u(t_0)\|_{\infty}$. Consider the unique solution $f_1$ to \eqref{SSODE} such that $f_1(0)=a$, $f_1'(0)=0$ given by Lemma \ref{lem.Q1} and let $\xi_1(a)$ be its first maximum point. Fix then $\xi_0(a)=\min\{\Xi/2,\xi_1(a)\}$, where $\Xi$ has been introduced in the statement of Lemma \ref{lem.interf}, and let $f_2$ be the unique decreasing solution to \eqref{SSODE} with interface at $\xi=\xi_0(a)$ and vertical asymptote as in \eqref{vert.asympt} given by Lemma \ref{lem.interf}. Since $f_1$ is increasing for $\xi\in(0,\xi_1(a))$, it is obvious that there exists an intersection point $\overline{\xi}(a)\in(0,\xi_1(a))$ between the two profiles such that $f_1(\overline{\xi}(a))=f_2(\overline{\xi}(a))$. Let us consider then the combined function
\begin{equation}\label{interm10}
f(\xi)=\left\{\begin{array}{ll}f_1(\xi), \qquad {\rm for} \ \xi\in[0,\overline{\xi}(a)],\\ f_2(\xi), \qquad {\rm for} \ \xi\in[\overline{\xi}(a),\xi_0(a)],\end{array}\right.
\end{equation}
and define the self-similar function
\begin{equation}\label{self.super}
\overline{U}(x,t;\tau)=(\tau+t_0-t)^{-\alpha}f(|x|(\tau+t_0-t)^{\beta}), \qquad x\in\real^N, \ t\in(t_0,t_0+\tau),
\end{equation}
with $\tau>0$ still a free parameter to be chosen later. By the construction, it is clear that $f(\xi)=\min\{f_1(\xi),f_2(\xi)\}$ for any $\xi\in(0,\xi_0(a))$, and since the self-similar solutions constructed following \eqref{SSS} with profiles $f_1$ and $f_2$ are solutions to Eq. \eqref{eq1} respectively on $[0,\xi_1(a)(\tau+t_0-t)^{-\beta}]$ and on $(0,\xi_0(a)(\tau+t_0-t)^{-\beta})$, we infer that $\overline{U}(\cdot,\cdot;\tau)$ defined in \eqref{self.super} is a compactly supported supersolution to Eq. \eqref{eq1} for any $\tau>0$. We are thus left with finding a parameter $\tau>0$ such that $\overline{U}(x,t_0;\tau)>u(x,t_0)$ for any $x\in\real^N$. To this end, it is enough to know that the support of $\overline{U}(\cdot,t_0;\tau)$ includes strictly the support of $u(t_0)$ and that the infimum of $\overline{U}(x,t_0;\tau)$ for $x\in[0,\zeta(t_0)]$ is larger than $\|u(t_0)\|_{\infty}$. The former condition leads to
\begin{equation}\label{cond.supp}
\zeta(t_0)<\xi_0(a)\tau^{-\beta},
\end{equation}
while the latter condition, together with the monotonicity of the function $f_1$, respectively $f_2$, on the corresponding regions considered in \eqref{interm10}, lead to
\begin{equation}\label{cond.height}
\begin{split}
\|u(t_0)\|_{\infty}&<\min\{\overline{U}(x,t_0;\tau): |x|\leq\zeta(t_0)\}=\min\{\overline{U}(0,t_0;\tau),\overline{U}(\zeta(t_0),t_0;\tau)\}\\
&=\tau^{-\alpha}\min\{a,f_2(\zeta(t_0)\tau^{\beta})\}.
\end{split}
\end{equation}
It is obvious that there exists $\tau<1$ sufficiently small such that both conditions \eqref{cond.supp} and \eqref{cond.height} are simultaneously fulfilled. Fixing such a value of $\tau>0$, we get that $\overline{U}(x,t_0;\tau)>u(x,t_0)$ for any $x\in\real^N$ and the comparison principle (see Theorem \ref{th.uniq}) then entails that $\overline{U}(x,t;\tau)>u(x,t)$ for any $t\in(t_0,t_0+\tau)$. In particular, we deduce from the form of $\overline{U}$ that $u(t)$ remains compactly supported for any $t\in(t_0,t_0+\tau)$. Since $t_0$ has been arbitrarily chosen (with the only condition that $u(t_0)$ is bounded, thus, before the finite blow-up time $T\in(0,\infty)$ of $u$), we complete the proof of Theorem \ref{th.fsp} in dimension $N\geq2$. The proof in dimension $N=1$ is totally similar, only that we have to consider a decreasing profile $f_2$ on $(0,\xi_0)$ with $f_2(0)=A>0$ sufficiently large, according to \eqref{neg.slope}.
\end{proof}

\section{Absence of localization}\label{sec.loc}

This section is devoted to the proof of Theorem \ref{th.loc}. To this end, we begin by borrowing an idea employed in dimension $N=1$ in \cite[Lemma 6]{FdPV06} and adapting it to our case. For a generic radius $R>0$, we consider the family of homogeneous Dirichlet problem for Eq. \eqref{eq1} in $B(0,R)$, that is,
\begin{equation}\label{HDP}
(HDP_R) \qquad \left\{\begin{array}{ll}w_t=\Delta w^m+|x|^{\sigma}w^p, & (x,t)\in B(0,R)\times(0,\infty),\\ w(x,t)=0, & (x,t)\in S(0,R)\times(0,\infty), \\
w(x,0)=w_0(x), & x\in B(0,R),\end{array}\right.
\end{equation}
where, as usual,
$$
S(0,R)=\{x\in\real^N: |x|=R\}.
$$
\begin{lemma}\label{lem.stationary}
For any $R>0$, there exists a stationary, radially symmetric solution $W_R$ to the problem \eqref{HDP} such that $W_R(0)>0$, $W_R'(0)=0$. Moreover, these stationary solutions are related through the following rescaling:
\begin{equation}\label{resc.stat}
W_R(r)=R^{(\sigma+2)/(m-p)}W_1\left(\frac{r}{R}\right), \qquad r=|x|, \qquad R>0.
\end{equation}
\end{lemma}
\begin{proof}
Let us remark that, by setting $r=|x|$, a radially symmetric and stationary solution to the problem \eqref{HDP} is in fact a solution to the differential equation
\begin{equation}\label{ODE}
(W^m)''(r)+\frac{N-1}{r}(W^m)'(r)+r^{\sigma}W^p=0, \qquad r\in(0,R),
\end{equation}
such that $W(R)=0$. Notice also at this point that a straightforward calculation gives that, if $W$ is a solution to \eqref{ODE} with $W(1)=0$, then the rescaling defined in \eqref{resc.stat} gives a solution to \eqref{ODE} with $W(R)=0$. In order to establish the existence of this family of solutions (which are all rescaled versions of a single one), we employ a technique of phase plane analysis. Let us thus introduce the variables
\begin{equation}\label{PSchange}
Y(\eta):=\frac{rW'(r)}{W(r)}, \qquad Z(\eta):=\frac{1}{m}r^{\sigma+2}W(r)^{p-m},
\end{equation}
with the new independent variable $\eta$ given by $\eta=\ln\,r$. We then get
$$
W'(r)=\frac{W(r)Y(r)}{r}, \qquad (W^m)'(r)=\frac{m}{r}W^m(r)Y(r),
$$
hence, by differentiating once more with respect to $r$ and taking into account that
$$
\frac{dY}{dr}=\frac{1}{r}\frac{dY}{d\eta},
$$
we easily find by direct calculation that
$$
(W^m)''(r)=\frac{m}{r^2}W^m(r)\left[\frac{dY}{d\eta}(r)+mY^2(r)-Y(r)\right].
$$
Replacing these expressions into \eqref{ODE} and passing to the variable $\eta$, we are left with the following autonomous dynamical system (where dot derivatives are derivatives with respect to $\eta$):
\begin{equation}\label{PPsyst}
\left\{\begin{array}{ll}\dot{Y}=-(N-2)Y-mY^2-Z, \\ \dot{Z}=Z(\sigma+2-(m-p)Y),\end{array}\right.
\end{equation}
a system that has been analyzed in detail in previous works \cite[Proposition 4.2]{IMS23} (for $m>1$, $p\in(1,m)$ and $\sigma\in(-2,0)$) and \cite[Section 5]{IMS23c} (for $m<1$, $p>1$). For the sake of completeness, we give some details of this analysis. Let us observe that the system \eqref{PPsyst} has two finite critical points, namely $P_0=(0,0)$ and $P_1=(-(N-2)/m,0)$, since we are only interested in the half-plane $\{Z\geq0\}$, by the definition of $Z$ in \eqref{PSchange}. It is easy to see (just by computing the matrix of the linearization) that if $N\geq3$, $P_0$ is a saddle point, while $P_1$ is an unstable node. Since unstable manifolds are unique (see \cite[Theorem 3.2.1]{GH}), we infer that there exists a unique orbit going out of $P_0$ on the (one-dimensional) unstable manifold of this saddle point.

Let us look next at this orbit. On the one hand, by integrating the linearized system of \eqref{PPsyst} in a neighborhood of the origin, we find that the manifold begins with
\begin{equation}\label{orbit}
Y(\eta)\sim-\frac{1}{N+\sigma}Z(\eta),
\end{equation}
which gives, in terms of the original function $W(r)$ (after undoing \eqref{PSchange} and integrating) a local behavior given by
$$
W(r)\sim\left[D-\frac{m-p}{(N+\sigma)(\sigma+2)}r^{\sigma+2}\right]^{1/(m-p)}, \qquad {\rm as} \ r\to0, \qquad D>0,
$$
which in particular fulfills the conditions
$$
W(0)=D^{1/(m-p)}>0, \qquad W'(0)=0.
$$
On the other hand, the unique trajectory in the unstable manifold of $P_0$ goes out into the negative half-plane $\{Y<0\}$ and remains there forever, since the flow of the system across the line $\{Y=0\}$ points towards the negative direction. Moreover, considering the isocline
\begin{equation}\label{iso}
Z=-(N-2)Y-mY^2,
\end{equation}
we observe that in a neighborhood of the origin, its slope is given by $Z/Y\sim-(N-2)$, while the slope of the trajectory \eqref{orbit} is given by $Z/Y\sim-(N+\sigma)$. Since $N+\sigma>N-2$, we observe that the trajectory \eqref{orbit} goes out into the upper region $\{Z>-(N-2)Y-mY^2\}$ of the negative half-plane $\{Y<0\}$. Taking the normal direction $\overline{n}=(-(N-2)-2mY,-1)$, we readily find that the flow of the system across the isocline \eqref{iso} is given by the sign of
$$
-Z(\sigma+2-(m-p)Y)<0, \qquad {\rm since} \ Y<0,
$$
hence the trajectory from $P_0$ cannot cross the isocline \eqref{iso} and has to stay forever in the region
$$
\mathcal{D}=\{Y<0, Z>-(N-2)Y-mY^2, Z>0\}.
$$
Notice next that, according to the two equations of the system \eqref{PPsyst}, we have $\dot{Y}<0$ and $\dot{Z}>0$ in the region $\mathcal{D}$, hence $Y(\eta)$ is decreasing with respect to $\eta$ and $Z(\eta)$ is increasing with respect to $\eta$ along all the trajectory. This shows that there exist
$$
Y_{\infty}:=\lim\limits_{\eta\to\infty}Y(\eta)<0, \qquad Z_{\infty}:=\lim\limits_{\eta\to\infty}Z(\eta)>0,
$$
and at least one of $Y_{\infty}$, $Z_{\infty}$ is not finite. Indeed, if both $Y_{\infty}$, $Z_{\infty}$ are finite, then the point $(Y_{\infty},Z_{\infty})$ would be critical for the system \eqref{PPsyst} according to Poincar\'e-Bendixon's theory \cite[Theorem 1, Section 3.7]{Pe}, and we know that the only finite critical points are $P_0$ and $P_1$. The previous arguments also prove that $Y(\eta)<0$ for any $\eta\in\real$ on the trajectory under study, hence
$$
\dot{Z}(\eta)=Z(\eta)(\sigma+2-(m-p)Y(\eta))>(\sigma+2)Z(\eta),
$$
which readily implies that $Z_{\infty}=+\infty$. We show next that $Y_{\infty}=-\infty$. Arguing by contradiction, suppose that $Y_{\infty}\in(-\infty,0)$. Since $Z(\eta)$ is a bijective function of $\eta\in\real$, the inverse function theorem allows us to express $Y$ as a function of $Z$ along the trajectory, and the system \eqref{PPsyst} then gives
$$
\frac{dY}{dZ}=-\frac{(N-2)Y+mY^2+Z}{Z[\sigma+2-(m-p)Y]}\to-\frac{1}{\sigma+2-(m-p)Y_{\infty}}\in(-\infty,0),
$$
which gives a linear behavior $Y=Y(Z)$ in the limit and contradicts the existence of such a vertical asymptote at $Y=Y_{\infty}\in(-\infty,0)$. We thus conclude that $Y_{\infty}=-\infty$. Coming then back to the system \eqref{PPsyst} and neglecting the lower order terms, we conclude that along the trajectory contained in the unstable manifold of $P_0$, we have
$$
\frac{dY}{dZ}\sim\frac{mY^2+Z}{(m-p)YZ}, \qquad {\rm as} \ Z\to\infty,
$$
or equivalently, after an integration and taking into account that $Z$ is bijective as a function of $\eta$,
\begin{equation}\label{interm12}
Y\sim-KZ^{m/(m-p)}, \qquad {\rm as} \ Z\to\infty, \qquad K>0.
\end{equation}
We may then conclude by following the end of the proof of \cite[Proposition 4.2]{IMS23} and deducing that this orbit connects to a critical attractor at infinity identified by $Y/Z\to-\infty$ as $Z\to\infty$, named $Q_3$ and analyzed in \cite[Lemma 3.4]{IMS23} (see also \cite[Lemma 2.6]{IS21}), which is characterized by a finite zero of the profiles $W(r)$ contained in the orbits entering it.

However, for the reader's convenience, we give here an alternative, direct argument for the existence of a finite zero. Assume for contradiction that there exists a profile $W$ contained in the unstable manifold of $P_0$ such that $W(r)>0$ for any $r\in[0,\infty)$. Since $W$ is decreasing with $r$, it follows that there exists $L\geq0$ such that $W(r)\to L$ as $r\to\infty$. Standard calculus arguments (see for example \cite[Lemma 2.9]{IL13}) prove that also $(W^m)'(r_n)\to0$ on a subsequence $r_n\to\infty$. Moreover, we infer from \eqref{interm12} and \eqref{PSchange} that
\begin{equation}\label{interm13}
\lim\limits_{r\to\infty}\frac{(W^m)'(r)}{r^{\theta-1}}=-Km^{-p/(m-p)}, \qquad \theta:=\frac{m(\sigma+2)}{m-p}>0.
\end{equation}
Since $\sigma>0$, we in fact have that $\theta>2$, and thus we reach a contradiction by evaluating the limit in \eqref{interm13} over the subsequence $r_n$ for which we established that $(W^m)'(r_n)\to0$. Thus, our stationary solution $W(r)$ has to have a finite zero.

We have thus proved that there exists at least one stationary, radially symmetric solution as in the statement having a finite zero, and the rescaling \eqref{resc.stat} shows that in fact we have a one-parameter family of such solutions, depending on the point where they attain the finite zero. The same is true in dimension $N=1$ or $N=2$, with the difference that the critical point $P_0$ is no longer a saddle (it becomes a saddle-node in dimension $N=2$ by unification between $P_0$ and $P_1$ according to \cite[Section 3.4]{GH}, and an unstable node in dimension $N=1$) but the trajectory starting as in \eqref{orbit} still exists in these dimensions. We omit the details here, some of them being given in \cite{IMS23}.
\end{proof}
With this construction, we are now in a position to prove Theorem \ref{th.loc}.
\begin{proof}[Proof of Theorem \ref{th.loc}]
The argument of the proof will be by contradiction, thus let us assume that there exists a continuous and compactly supported initial condition $u_0$ as in \eqref{icond} such that the (unique) solution to the Cauchy problem \eqref{eq1}-\eqref{ic} remains localized, in the sense that there exists $R_0$ sufficiently large such that $u(x,t)=0$ for any $x\in\real^N\setminus B(0,R_0)$ and $t\in[0,T]$, where $T$ is the blow-up time of $u$ according to Theorem \ref{th.BU}. Consider $R>0$ sufficiently large such that the radially symmetric, stationary solution $W_R$ to the problem \eqref{HDP} given by \eqref{resc.stat} satisfies $W_R(x)\geq u_0(x)$, for any $x\in\real^N$. This is possible, and in fact, since $W_1$ is decreasing, in order to be fulfilled, it suffices to have $R>R_0$ and
$$
W_R(r)>R^{(\sigma+2)/(m-p)}W_1\left(\frac{R_0}{R}\right)\geq\|u_0\|_{\infty}, \qquad r\in[0,R_0]
$$
which obviously holds true, provided $R$ is taken sufficiently large. With this choice of $R$, and due to the localization of the solution $u$ in $B(0,R_0)\subset B(0,R)$, we find that both $u$ and $W_R$ are solutions to the Dirichlet problem \eqref{HDP} on $B(0,R)$. But, since $\sigma>0$, we have $|x|^{\sigma}\in L^{\infty}(B(0,R))$ and the comparison principle for the Dirichlet problem \cite[Proposition 2.2]{Su02} then entails that
$$
u(x,t)\leq W_R(r), \qquad x\in B(0,R), \ t\in[0,T].
$$
We have thus reached a contradiction with the fact that $u$ blows up at time $T$, since the stationary solution $W_R$ is always bounded, ending the proof.
\end{proof}

\bigskip

\noindent \textbf{Acknowledgements} R. G. I. and A. S. are partially supported by the Grants PID2020-115273GB-I00 and RED2022-134301-T (Spain). M. L is partially supported by the Grant PID2022-136589NB-I00, all grants funded by funded by MCIN/AEI/10.13039/501100011033.

\bibliographystyle{plain}

\end{document}